\documentclass{amsart}
\usepackage[utf8]{inputenc}
\usepackage{mathtools, amsmath, amssymb}
\usepackage{ytableau}

\numberwithin{equation}{section}

\theoremstyle{plain}
\newtheorem{theorem}{Theorem}[section]
\newtheorem{lemma}{Lemma}[section]
\newtheorem{corollary}{Corollary}[section]

\newcommand{\XX}{{\mathcal{X}}}
\newcommand{\spechtj}[2]{{S^{(m-#1, #2)(m-#1, #1)}}}
\newcommand{\specht}[1]{{S^{(m-#1, #1)}}}
\newcommand{\spechtd}[1]{{D^{(m-#1, #1)}}}
\newcommand{\Mperm}[1]{{M^{(m-#1, #1)}}}
\newcommand{\allone}{{\mathbf{1}}}

\DeclareMathOperator{\ch}{char}
\DeclareMathOperator{\im}{Im}
\DeclareMathOperator{\GL}{GL}

\title[$1$-intersection ranks of $2$-subsets vs. $n$-subsets]{A representation-theoretic computation of the rank of $1$-intersection incidence matrices: $2$-subsets vs. $n$-subsets.}
\author[Ducey]{Joshua E. Ducey}
\author[Sherwood]{Colby J. Sherwood}
\address{Dept.\ of Mathematics and Statistics, James Madison University, Harrisonburg, VA 22807, USA}
\email{duceyje@jmu.edu}
\email{sherwocj@dukes.jmu.edu}
\begin{document}
\begin{abstract}
    Let $W_{k,n}^{i}(m)$ denote a matrix with rows and columns indexed by the $k$-subsets and $n$-subsets, respectively, of an $m$-element set.  The row $S$, column $T$ entry of $W_{k,n}^{i}(m)$ is $1$ if $|S \cap T| = i$, and is $0$ otherwise.  We compute the rank of the matrix $W_{2,n}^{1}(m)$ over any field by making use of the representation theory of the symmetric group.  We also give a simple condition under which $W_{k,n}^{i}(m)$ has large $p$-rank.
\end{abstract}
\maketitle

\section{Introduction}
Incidence matrices are interesting objects that encode a relation between two finite sets into a zero-one matrix that can then be studied algebraically.  Properties of the matrix that are unchanged by the different ways the matrix can be constructed become properties of the relation itself.  One example of such an invariant is the rank of the matrix, or more generally the elementary divisors of the matrix.  In the case that the two sets are the same and the relation describes adjacency of vertices in a graph, the spectrum of the matrix is another invariant.  In this paper we will be interested in the rank.  Note that the incidence matrix can be defined over any field, and the answer can depend on the field's characteristic.

Given nonnegative integers $k \leq n \leq m$ and a fixed set $\XX$ of size $m$, one can define an $\binom{m}{k} \times \binom{m}{n}$ incidence matrix $W_{k,n}^{i}(m)$ as follows.  Let the rows of $W_{k,n}^{i}(m)$ be indexed by the $k$-element subsets ($k$-subsets, for short) of $\XX$, the columns be indexed by the $n$-subsets of $\XX$, and let the row $S$, column $T$ entry of $W_{k,n}^{i}(m)$ be $1$ if $|S \cap T| = i$, and $0$ otherwise.  These subset-intersection matrices describe fundamental relations and are naturally interesting.

Suppose for a moment that we have $i = k$; the matrix $W_{k,n}^{k}(m)$ is the well-studied \textit{inclusion matrix} of $k$-subsets vs.\ $n$-subsets. Mathematicians have been interested in this matrix since the 1960s, when it was shown that $W_{k,n}^{k}(m)$ has full rank over the rational numbers \cite{gottlieb}.  Later, the rank of $W_{k,n}^{k}(m)$ was calculated over any field of characteristic $2$ (the $2$-rank, for short) and the $3$-rank was computed when $n=k+1$ \cite{LR}.  The problem was completely solved when Wilson \cite{wilson:diagonal} found a beautiful diagonal form for the inclusion matrices (the $p$-rank is the number of diagonal entries of Wilson's form not divisible by $p$).  Further refinements to Wilson's result were made in \cite{frankl} and \cite{bier}.  

The study of inclusion matrices has applications to the theory of designs.  For much information and more history of these matrices see \cite[Section 10]{qing} and \cite{plaza}.  For a recent application to computing integer invariants of the $n$-cube graph, see \cite{CXS:cube}.

We now switch our attention to the situation where $i < k$.  When $i=0$ the matrices $W_{k,n}^{0}(m)$ describe the \textit{disjointness} relation.  However note that a $k$-subset $S$ being disjoint from an $n$-subset $T$ happens precisely when $S \subseteq \XX \backslash T$.  So we have $W_{k,n}^{0}(m) = W_{k,m-n}^{k}(m)$ and thus the ranks of the disjointness matrices are also known.

When $0<i<k$ much less is known.  In \cite[Example SNF3]{eijl} the authors solve the rank problem for the matrix $W_{2,2}^{1}(m)$ by finding a diagonal form (the Smith normal form in that case).  The matrix $W_{2,2}^{1}(m)$ can also be thought of as an adjacency matrix of the triangular graph $T(m)$, and if one instead considers the Laplacian matrix the rank problem has been solved for both $T(m)$ \cite{line} and its complement (the Kneser graph on $2$-subsets) \cite{kneser2}.  

In \cite{wong}, a diagonal form is found for a very general class of matrices $N_{2}(G)$.  The rows of $N_{2}(G)$ are indexed by the $2$-subsets of a set of size $m$, the first column is the characteristic vector of the edges of a graph $G$ on $m$ vertices, and the rest of the columns come from the action of the symmetric group on this first column.  When $G$ is the complete bipartite graph $K_{n,m-n}$ the matrix $N_{2}(G)$ becomes $W_{2,n}^{1}(m)$, and from the result \cite[Theorem 17]{wong} the rank of $W_{2,n}^{1}(m)$ follows.  

The purpose of this paper is to give an alternative computation of the rank of $W_{2,n}^{1}(m)$ by using the representation theory of the symmetric group $G_{m}$.  There are several reasons why one would want to do this.

Representation theory has already been shown to be a powerful tool for problems such as computing $p$-ranks and Smith normal forms.  Some successes in this respect, including the $q$-analogue of this problem (that is, concerning various intersection relations of subspaces of a vector space and the representation theory of $\GL(n,q)$) include \cite{1spaces, skew, frumkin, kneser2, jolliffe, raghu}.  Whether intentionally or not, in the various statements and hypotheses of theorems in the works on subsets one can find standard Young tableau, proper partitions, and other reflections of the representation theory of $G_{m}$ present.  There has been particular interest in recasting the impressive matrix methods (fronts, shadows, etc.) of Wilson and Wong \cite{wilson:diagonal, wong} in the hope that some light might be shed on other incidence problems of subsets.  The recent paper \cite{jolliffe} makes progress on this for the inclusion/disjointness matrices $W_{k,n}^{0}(m)$; we now give attention to the $1$-intersection matrices of $2$-subsets vs. $n$-subsets.

Furthermore, representation theory provides an organized setting to frame other related problems.  Incidence matrices of subset inclusion, disjointness, intersection in a fixed size; the Laplacian matrices (and others) of Kneser-type graphs; these all represent $FG_{m}$-module homomorphisms, and understanding of ranks (or Smith normal forms) of these matrices can come from a sufficient understanding of the module structure of the domain and codomain.  In particular, the incidence problems concerning $3$-subsets of a set seem to have not been handled before.  Although the complexity of the modules will increase, the manner of investigation remains the same.

The papers \cite{frumkin, jolliffe, kneser2} make use of the modular representation theory techniques expounded by James \cite{james}.  This is the approach we will take here.  The next section will review the basic ideas that we need.  Section \ref{sec:calc} will provide some useful calculations that will be used repeatedly, and in Section \ref{sec:cases} we state and prove our theorem.  In the final Section \ref{sec:conclude} we give a lemma that may be used to investigate   the general  subset-intersection matrix $W_{k,n}^{i}(m)$.  We also identify a simple condition that forces this matrix to have large rank.

\section{Representation Theory of Symmetric Groups}
Suppose $\lambda = (\lambda_1, \lambda_2, \lambda_3,...)$ is a partition of $m$. That is, $\lambda_1 \geq \lambda_2 \geq \lambda_3 \geq ...\geq 0$ and $\sum_k \lambda_k = m$. A $\lambda$-tableau is an array of the integers from $1$ to $m$ without repeats where the $k$-th row has $\lambda_k$ entries. For example, if $\lambda = (3,2,1)$, the following are $\lambda$-tableaux.

\vspace{.2 in}
\begin{center}
\begin{ytableau}
1 & 2 & 3 \\
4 & 5 \\
6
\end{ytableau}
\hspace{.2 in}
\begin{ytableau}
4 & 2 & 3 \\
6 & 1 \\
5
\end{ytableau}
\hspace{.2 in}
\begin{ytableau}
6 & 1 & 5 \\
2 & 4 \\
3
\end{ytableau}
\end{center}
\vspace{.2 in}

For a given $\lambda$-tableau, $t$, we define its \textit{row stabilizer}, $R_t$ to be the set of $\sigma \in G_m$, where $G_m$ is the symmetric group on $m$ elements, that keep the rows of $t$ fixed set-wise. We now define the tabloid, $\{t\}$, to be the equivalence class of $t$ under the relation $t_1 \sim t_2$ if  $t_1 = \sigma t_2$ for some $\sigma \in R_{t_1}$. We now as in \cite{jolliffe} define the $j$-\textit{column stabilizer} of $t$, denoted $C_t^j$, to be the set of permutations that fix each of the first $j$ columns of $t$ set-wise, and fix the remaining symbols in $t$ point-wise.  Note that when $j \geq \lambda_2$, the $j$-column stabilizer fixes all columns set-wise; it is then called the \textit{column stabilizer} of $t$ and is denoted $C_t$.

Now fix a field $F$ and consider the group algebra $FG_{m}$.  Let $M^{\lambda}$ denote the $FG_{m}$-permutation module with basis the $\lambda$-tabloids.  
Define the $j$-\textit{polytabloid} for tableau $t$, denoted $e_t^j$, as follows:
$$e_t^j = \kappa_t^j\{t\}$$ where
$$\kappa_t^j = \sum_{\sigma \in C_t^j} (-1)^{\sigma}\sigma$$
and 
$(-1)^{\sigma}$ is $1$ or $-1$ when $\sigma$ is even or odd, respectively.

When $C_t^j = C_t$, we call a $j$-polytabloid $e_{t}^{j}$  a \textit{polytabloid} and denote it by $e_{t}$. For example, let us consider the following tableau for the partition $\lambda=(4,2)$.

\vspace{.2 in}
$t = $
\begin{ytableau}
1 & 2 & 3 & 4 \\
5 & 6 \\
\end{ytableau}

\vspace{.2 in}
$\{t\} =$
\ytableausetup
{boxsize=normal,tabloids}
\ytableaushort{
1234, 56
}
$=$
\ytableaushort{
2134, 65
}
$=$
\ytableaushort{
1243, 56
}
\vspace{.2 in}

$e_t^0 = \{t\}$

\vspace{.2 in}
$e_t^1 =$
\ytableaushort{
1234, 56
}
$-$
\ytableaushort{
5234, 16
}

\vspace{.2 in}
$e_t^2 = e_t =$
\ytableaushort{
1234, 56
}
$-$
\ytableaushort{
5234, 16
}
$-$
\ytableaushort{
1634, 52
}
$+$
\ytableaushort{
5634, 12
}

\vspace{.2 in}

It is important to realize that the $j$-polytabloid for $t$ depends on the tableau $t$, not the tabloid $\{t\}$.


For our purposes we will restrict our attention to partitions of the form $\lambda = (m-i, i)$ where $i \leq m-i$. Any $(m-i,i)$-tabloid is determined by the $i$-subset of $\XX = \{1, 2, \cdots, m\}$ in its second row, and for convenience of notation we will often identify the two. 

Let $S^{(m-i, j)(m-i,i)}$ be the submodule of $M^{(m-i,i)}$ spanned by $j$-polytabloids. Notice that $\specht{i} \coloneqq S^{(m-i,i)(m-i, i)}$ is the span of polytabloids; this is the \textit{Specht module} for the partition $(m-i,i)$.   We have the following filtration of $\Mperm{i}$:
\[
M^{(m-i,i)} \supseteq S^{(m-i,1)(m-i,i)} \supseteq S^{(m-i,2)(m-i,i)} \supseteq \cdots \supseteq S^{(m-i,i)(m-i,i)}=S^{(m-i,i)}\supseteq 0.
\]
The reader can consult \cite{james}, from which our notation is taken, for the general theory.  It can also be  shown  (\cite[Chapter 17]{james}) that successive quotients of this submodule chain are isomorphic to Specht modules. That is, 
\begin{equation}\label{eqn:spechtseries}
S^{(m-i, j)(m-i,i)}/S^{(m-i,j+1)(m-i,i)} \cong S^{(m-j, j)}.
\end{equation}

The Specht modules above can be defined over any field $F$.  When $\ch(F) = 0$ they are irreducible, but this is not always the case when $F$ has positive characteristic.  When the indexing partition $\lambda = (m-i,i)$ has two parts and $2i < m$, it turns out that the Specht module $\specht{i}$ has a unique maximal submodule with simple quotient denoted $\spechtd{i}$.  Furthermore, the only other possible composition factors of $\specht{i}$ are $\spechtd{j} (0 \leq j < i)$, and these occur with multiplicity zero or one depending on both the characteristic $p$ and $m$.  We collect this important information below, which is a special case of Theorem 24.15 in \cite{james}.

\begin{theorem}[\cite{james}, Theorem 24.15]\label{thm:james}
Let $F$ be a field with $\ch(F) = p$.  Suppose $m>2i$ and let $\spechtd{i}$ denote the simple head of the Specht module $\specht{i}$ defined over $F$.  Let $[\specht{i}:\spechtd{j}]$ denote the multiplicity of $\spechtd{j}$ as a composition factor of $\specht{i}$.  Then
\begin{enumerate}
    \item $S^{(m)} = D^{(m)}$.
    \item $[\specht{1} : D^{(m)}] = 1$ if $m \equiv 0 \pmod{p}$, $0$ otherwise.
    \item $[\specht{2}:\spechtd{1}] = 1$ if $m \equiv 2 \pmod{p}$, $0$ otherwise.
    \item Let $p>2$.  Then $[\specht{2} : D^{(m)}] = 1$ if $m \equiv 1 \pmod{p}$, $0$ otherwise.
    \item Let $p=2$.  Then $[\specht{2} : D^{(m)}] = 1$ if $m \equiv 1 \mbox{ or } 2 \pmod{4}$, $0$ otherwise.
\end{enumerate}
\end{theorem}

\section{Some calculations}\label{sec:calc}

Let $0 \leq k \leq n \leq \frac{1}{2}m$ and consider the map $\rho_{n,k} \colon M^{(m-n, n)} \rightarrow M^{(m-k, k)}$ that sends an $(m-n,n)$-tabloid $\{t\}$ to the sum of the $(m-k,k)$-tabloids that each have in their second row exactly one symbol in common with the second row of $\{t\}$.  In other words, an $n$-subset maps to the sum of the $k$-subsets that intersect it in a set of size $1$; we may consider this to be the map defined by the $1$-intersection matrix $W_{k,n}^{1}(m)$.  Fixing a field $F$, the domain and codomain are $FG_{m}$-permutation modules and we will compute the rank of this $FG_{m}$-homomorphism when $k=2$.  

When a $2$-subset intersects an $n$-subset in a set of size $1$, it also intersects the complement of the $n$-subset in a set of size $1$.  Thus we have $W_{2,n}^{1}(m) = W_{2,m-n}^{1}(m)$, and so we lose  no generality by our assumption that $n \leq \frac{1}{2}m$.  We will make this assumption throughout the paper.

Let  $\psi_{k,j} \colon M^{(m-k, k)} \rightarrow M^{(m-j, j)}$ denote the \textit{inclusion} map; that is, the map that sends an $(m-k,k)$-tabloid $\{s\}$ to the sum of the $(m-j,j)$-tabloids that each have as their second row a subset of the second row of $\{s\}$.  The two lemmas below collect some useful calculations that will be used repeatedly.  In Lemma \ref{lem:maps3} we give a more general statement of the lemma below that applies to all the incidence matrices $W_{k,n}^{i}(m)$; however, for clarity, here we state and prove the special case related to the maps we are considering.

We again remind the reader that for notational convenience we often identify an $(m-k,k)$-tabloid with the $k$-subset of elements in its second row.
\begin{lemma}\label{lem:maps1}
Suppose $2 \leq n \leq \frac{1}{2}m$.  Let $t = 
\ytableausetup{smalltableaux, notabloids}
\begin{ytableau}
b & d &  \cdots & & \cdots & \\
a & c &  \cdots & 
\end{ytableau}$ be an $(m-n,n)$-tableau.
\begin{enumerate}
    \item \label{eqns:map0} 
    $e_{t}^{0} = \{t\} \xmapsto{\rho_{n,2}} \sum_{|\{s\}\cap \{t\}|=1} \{s\} \xmapsto{\psi_{2,0}} n(m-n)\emptyset$.
    \item 
    $e_t^{1} \xmapsto{\rho_{n,2}}
    \sum_{\{s\} \in \mathcal{A}} e_{s}^{1} + \sum_{\{s^{\prime}\} \in \mathcal{B}} e_{s^{\prime}}^{1} \xmapsto{\psi_{2,1}} (m-2n)e_{\begin{ytableau}
b &  & & & \cdots & \\
a 
\end{ytableau}}$
    , \\where $\mathcal{A}$ consists of the $2$-subsets containing $a$ but not $b$, $\mathcal{B}$ consists of the $2$-subsets containing $b$ but not $a$, and $s$ (resp. $s^{\prime}$) is a tableau representing such a $2$-subset with $b$ (resp. $a$) in the top-left position.
    \item $e_{t}^{2} \xmapsto{\rho_{n,2}} 2e_{\begin{ytableau}
b & c & & & \cdots & \\
a & d
\end{ytableau}}$.

\end{enumerate}
\end{lemma}

\begin{proof}
For part $(1)$ of the lemma, notice that the inclusion map $\psi_{2,0}$ sends each $2$-subset to the empty set (that is, the tabloid with only one row).  The number of $2$-subsets sharing exactly one element with the $n$-subset $\{t\}$ is $n(m-n)$.  

For part $(2)$, we have 
\begin{equation}
\rho_{n,2}(e_{t}^{1}) = \kappa_{t}^{1}\rho_{n,2}(\{t\}) = \kappa_{t}^{1}\left ( \sum_{|\{s\}\cap \{t\}|=1} \{s\} \right )
\end{equation}
and since $\kappa_{t}^{1} = \left( 1 - (ab) \right)$, the only terms in this sum not killed by $\kappa_{t}^{1}$ are those $2$-subsets in $\mathcal{A}$ or $\mathcal{B}$.  If we represent such a $2$-subset with a tableau $s$ that has both $a$ and $b$ in the first column, then $\kappa_{t}^{1} = \kappa_{s}^{1}$ and so 
\begin{equation}\label{eqn:sum}
\rho_{n,2}(e_{t}^{1}) = \sum_{\{s\} \in \mathcal{A}} e_{s}^{1} + \sum_{\{s^{\prime}\} \in \mathcal{B}} e_{s^{\prime}}^{1}.
\end{equation}
Notice that if $s = \begin{ytableau}
b & c & & & \cdots & \\
a & d
\end{ytableau}$ then we have (omitting the first row of the tabloid):
\[
e_{s}^{1} = \overline{ad} - \overline{bd}
\]
and $\psi_{2,1}(e_{s}^{1}) = \overline{a} - \overline{b}= e_{\begin{ytableau}
b &  & & & \cdots & \\
a 
\end{ytableau}}$.  Thus each term in the first sum of Equation \ref{eqn:sum} is mapped by $\psi_{2,1}$ to $e_{\begin{ytableau}
b &  & & & \cdots & \\
a 
\end{ytableau}}$ and in the same way one sees that each term in the second sum maps to $e_{\begin{ytableau}
a &  & & & \cdots & \\
b 
\end{ytableau}}= -e_{\begin{ytableau}
b &  & & & \cdots & \\
a 
\end{ytableau}}.$  Part $(2)$ of the lemma now follows by counting the number of terms in each sum.

Finally, to prove part $(3)$ of the lemma, notice that $\kappa_{t}^{2} = (1-(ab))(1-(cd))$.  Therefore
\[
\rho_{n,2}(e_{t}^{2}) = (1-(ab))(1-(cd))\left ( \sum_{|\{s\}\cap \{t\}|=1} \{s\} \right ).
\]
One easily sees that the only terms in this sum not killed by $\kappa_{t}^{2}$ are $\overline{ad}$ and $\overline{bc}$.  Thus
\begin{align*}
\rho_{n,2}(e_{t}^{2}) &= \kappa_{t}^{2}\left( \overline{ad} + \overline{bc} \right)\\
&= e_{\begin{ytableau}
b & c & & & \cdots & \\
a & d
\end{ytableau}} + e_{\begin{ytableau}
a & d & & & \cdots & \\
b & c
\end{ytableau}}\\
&= 2e_{\begin{ytableau}
b & c & & & \cdots & \\
a & d
\end{ytableau}}
\end{align*}

\end{proof}
The lemma above gives information about the images under $\rho_{n,2}$ of $j$-polytabloids in $\Mperm{n}$.  The next lemma describes what the inclusion maps $\psi_{2,j}$ do to the image of $\rho_{n,2}$ in $\Mperm{2}$.  
\begin{lemma}\label{lem:maps2}
Let $\{t\}$ be an $n$-subset of $\XX$.  Set $Y = \rho_{n,2}(\{t\})$.  Note we may view $Y$ as the element of $\Mperm{2}$ represented by a column of $W_{2,n}^{1}(m)$.
\begin{enumerate}
    \item $\psi_{2,0}(Y) = n(m-n)\emptyset$.
    \item $\psi_{2,1}(Y) = (m-n)\sum_{\{z\} \subset \{t\}} \{z\} + n\sum_{\{z^{\prime}\} \not\subset \{t\}} \{z^{\prime}\}$.
\end{enumerate}
\end{lemma}
\begin{proof}
Part (1) is just a restatement of part $(1)$ of Lemma \ref{lem:maps1}.  To see part $(2)$, notice that if a $1$-subset $\{z\}$ is contained in the $n$-subset $\{t\}$, then there are $(m-n)$ $2$-subsets that simultaneously contain $\{z\}$ and meet $\{t\}$ in a singleton (that singleton necessarily being $\{z\}$).  If $\{z\} \not\subset \{t\}$, then there are $n$ $2$-subsets that contain $\{z\}$ and meet $\{t\}$ in a singleton.
\end{proof}


One more result we need is a consequence of the well-known hook length formula \cite[Theorem 20.1]{james}. 
\begin{theorem}\label{thm:hook}
Over any field, the dimension of the Specht module $\specht{j}$ is 
\[
\binom{m}{j} - \binom{m}{j-1},
\]
where $\binom{m}{-1}=0$.
\end{theorem}

We now describe how we will compute the rank of $W_{2,n}^{1}(m)$.  Let $F$ be a field. The matrix $W_{2,n}^{1}(m)$ represents the $FG_{m}$-module homomorphism 
\[
\rho_{n,2} \colon \Mperm{n} \to \Mperm{2},
\]
so we will compute the dimension of $\im{\rho_{n,2}}$.  As mentioned previously, there is a Specht series 
\[
\Mperm{2}  \supseteq \spechtj{2}{1} \supseteq \spechtj{2}{2} = \specht{2} \supseteq \{0\}
\]
with successive quotients isomorphic to Specht modules.  To be precise, the inclusion $\psi_{2,j}$ maps $\spechtj{2}{j}$ onto $\specht{j}$, and $\ker{\psi_{2,j}} \cap \spechtj{2}{j} = \spechtj{2}{j+1}$.  Following the idea in \cite{jolliffe}, we set $P^{j} \coloneqq \im{\rho_{n,2}} \cap \spechtj{2}{j}$, for $0 \leq j \leq 2$.  Then 
\[
\im{\rho_{n,2}} = P^0 \supseteq P^{1} \supseteq P^{2} \supseteq \{0\}
\]
is a filtration for $\im(\rho_{n,2})$ and by the second isomorphism theorem each quotient $L^{j} \coloneqq P^{j}/P^{j+1}$ is isomorphic to a submodule of $\specht{j}$.  We indicate this situation with the picture
\begin{center}
\begin{tabular}{lll}
  &                     & $L^{0}$ \\
$\im{\rho_{n,2}}$ & $\sim$ & $L^{1}$   \\
  &                     & $L^{2}$
\end{tabular}.
\end{center}


It turns out that if we replace $\rho_{n,2}$ with the inclusion map $\psi_{n,2}$, each of these quotients is either zero or the full Specht module \cite{jolliffe}.  The situation with size-$1$ intersection is a bit more subtle since $L^{j}$ can be a proper submodule of $\specht{j}$.  However, we will be able to work out what is going on with the help of Theorem \ref{thm:james} that describes the composition factors of $\specht{j}$.

We emphasize once more that what Lemma \ref{lem:maps1} tells us is that $j$-polytabloids in $\Mperm{n}$ are mapped by $\rho_{n,2}$ to elements of $\spechtj{2}{j}$ in $\Mperm{2}$.  These images are representatives of elements of $L^{j}$, and we can identify the element by further applying $\psi_{2,j}$.

\section{The Main Result}\label{sec:cases}

Let $2 \leq n \leq m$. Recall that we may assume $2n \leq m$.  Here is the main result of this paper.

\begin{theorem}
Let $2 \leq n \leq \frac{1}{2}m$.  Let $F$ be a field and view the entries of $W_{2,n}^{1}(m)$ as coming from $F$.   The rank of $W_{2,n}^{1}(m)$ is given in the following table:
\begin{center}
\begin{tabular}{l|l}
    Case & Rank \\
    \hline
$\ch{F} =0$, $2n<m$     & $m(m-1)/2$\\
$\ch{F} = 0$, $2n=m$ & $(m-1)(m-2)/2$\\
$\ch{F}=2$, $m$ odd & $m-1$\\
$\ch{F} = 2$, $m$ even, $n$ even & $m-2$\\
$\ch{F} = 2$, $m$ even, $n$ odd & $m-1$\\
$\ch{F} = p>2$, $p \nmid m-2n$, $p \nmid n(m-n)$ & $m(m-1)/2$\\
$\ch{F} = p>2$, $p \nmid m-2n$, $p \mid n(m-n)$ & $(m+1)(m-2)/2$\\
$\ch{F} = p>2$, $p \mid m-2n$, $p \mid m$ & $m(m-3)/2$\\
$\ch{F} = p>2$, $p \mid m-2n$, $p \nmid m$ & $(m-1)(m-2)/2$.
\end{tabular}
\end{center}
\end{theorem}

We will spend the rest of this section proving the theorem.
  For our first and simplest case we consider:

\subsection{Case: $\ch{F}=0$ and $2n<m$.} \hfil \\

Since all of $n(m-n)$, $m-2n$, and   $2$ are nonzero, we see from Lemma \ref{lem:maps1} that each of the $L^{j}$ contains a polytabloid that generates the full Specht module.  By Theorem \ref{thm:hook} the dimension of $\im{\rho_{n,2}}$ is 
\begin{align*}
&\dim{L^{0}} + \dim{L^{1}} + \dim{L^{2}}\\ &= \dim{S^{(m)}} + \dim{\specht{1}} + \dim{\specht{2}}\\
&= \binom{m}{0} - \binom{m}{-1}+\binom{m}{1} - \binom{m}{0}+\binom{m}{2} - \binom{m}{1}\\
&=\binom{m}{2}\\
&= \frac{m(m-1)}{2}
\end{align*}
and so in this case $W_{2,n}^{1}(m)$ has full rank.

\subsection{Case: $\ch{F}=0$ and $2n=m$.}\hfil \\

A look at Lemma \ref{lem:maps1} shows that $L^{0} \cong S^{(m)}$ and also $L^{2} \cong \specht{2}$.  Let us try to identify $L^{1}$.  From Lemma \ref{lem:maps2} we see that 
\[
L^{1} \cong \psi_{2,1}\left(\im{\rho_{n,2}} \cap \spechtj{2}{1}\right) \subseteq \langle \allone \rangle,
\]
where $\allone$ denotes the all-one vector; that is, the sum of all tabloids.  Since in this case $\specht{1}$ is irreducible, we must in fact have
\[
\psi_{2,1}\left(\im{\rho_{n,2}} \cap \spechtj{2}{1}\right) =  0.
\]
Thus $W_{2,n}^{1}(m)$ has rank 
\[
1 + \binom{m}{2} - \binom{m}{1} = \frac{(m-1)(m-2)}{2}.
\]

\subsection{Case: $\ch{F} = 2$.}\hfil \\

In this case, part $(3)$ of Lemma \ref{lem:maps1} shows that $\spechtj{n}{2}$ is in the kernel of $\rho_{n,2}$.  Thus by considering the Specht series in $\Mperm{n}$ one sees that the composition factors of $\im{\rho_{n,2}}$ must be a sub-multiset of the composition factors of $S^{(m)}$ and $\specht{1}$.  Always we have that $S^{(m)} = D^{(m)}$ is the trivial module and $\specht{1}$ has $\spechtd{1}$ as a top composition factor, but whether or not $\specht{1}$ contains $D^{(m)}$ as a composition factor will depend on the parity of $m$.

\subsubsection{Case: $m$ even.}\hfil \\

By Theorem \ref{thm:james}, we know that $\specht{1}$ has $\{\spechtd{1}, D^{(m)}\}$ as its set of composition factors.  Notice that this implies that $\dim \spechtd{1} = \binom{m}{1} - \binom{m}{0} - 1$.
Thus in this case the composition factors of $\im{\rho_{n,2}}$ are a sub-multiset of 
\[
\{D^{(m)},\spechtd{1}, D^{(m)}\}.
\]

In fact, one can deduce (\cite[Chapter $5$, Example $1$]{james}) that $\Mperm{1}$ is uniserial with unique composition series
\[
\Mperm{1} \supset \specht{1} \supset \langle \allone \rangle \supset \{0\}
\]
and corresponding composition factors $D^{(m)}, \spechtd{1}, D^{(m)}$.  We will make use of this very understandable submodule structure of $\Mperm{1}$ in the following way.  Since $\ch{F} = 2$, it is easy to see that
\[
\rho_{n,2} = \psi_{1,2} \circ \psi_{n,1},
\]
where $\psi_{1,2}$ sends a $1$-subset to the sum of the $2$-subsets containing it.  This means that $\im{\rho_{n,2}}$ is isomorphic to a subquotient of $\Mperm{1}$.

The image of $\psi_{n,1} \colon \Mperm{n} \to \Mperm{1}$ is easily identified with the same technique we have been applying to $\rho_{n,2}$; namely, intersect $\im{\psi_{n,1}}$ with the Specht series in $\Mperm{1}$ and focus on the successive quotients.  The picture is
\begin{center}
\begin{tabular}{lll}
$\im{\psi_{n,1}}$ & $\sim$ & $N^{0}$   \\
  &                     & $N^{1}$
\end{tabular}.
\end{center}

This is done in \cite{jolliffe} for all the inclusion maps, and always the quotients are either zero or the entire Specht module.  The specific result applied to this case is that $N^{j} \cong \specht{j}$ precisely when $\binom{n-j}{1-j}$ is odd.  Thus the answer is this case depends on the parity of $n$.

\paragraph{Case: $n$ even.} \hfil \\

  Since $n$ is even the discussion above shows that  $\im{\psi_{n,1}}$ is $\specht{1}$.

Furthermore, it is easy to see that $\allone$ (which generates  $D^{(m)}$ in $\specht{1}$) is in the kernel of $\psi_{1,2}$.  It follows that $\im{\rho_{n,2}} \cong \spechtd{1}$, and so in this case the rank of $W_{2,n}^{1}(m)$ is
\[
\binom{m}{1} - \binom{m}{0} - 1 = m-2.
\]

\paragraph{Case: $n$ odd.} \hfil \\

In this case we see that $\psi_{n,1} \colon \Mperm{n} \to \Mperm{1}$ is surjective.  Again $\allone \in \ker{\psi_{1,2}}$, and so we see that the composition factors of $\im{\rho_{n,2}}$ are a subset of $\{D^{(m)}, \spechtd{1}\}$.  We pick up $D^{(m)}$ as a top composition factor of $\im{\rho_{n,2}}$ from Lemma \ref{lem:maps1} part $(1)$.  Clearly $G_{m}$ does not act trivially on $\im{\rho_{n,2}}$, so we must pick up $\spechtd{1}$ as a composition factor as well.  Therefore in this case the rank of $W_{2,n}^{1}(m)$ is
\[
\binom{m}{1}-\binom{m}{0} = m-1.
\]

\subsubsection{Case: $m$ odd.}\hfil \\

By Theorem \ref{thm:james}, we have $\specht{1} = \spechtd{1}$.  Thus the composition factors of $\im{\rho_{n,2}}$ form a subset of $\{ D^{(m)}, \spechtd{1}\}$.   As in the previous case, 
\[
\rho_{n,2} = \psi_{1,2} \circ \psi_{n,1},
\]
and so $\im{\rho_{n,2}}$ is a subquotient of $\Mperm{1}$.  Under the current hypotheses, it is easily deduced  (\cite[Chapter $5$, Example $1$]{james}) that 
\[
\Mperm{1} = \langle \allone \rangle \oplus \specht{1}.
\]
Since $\allone \in \ker{\psi_{1,2}}$, and $\rho_{n,2} \neq 0$, we must have $\im{\rho_{n,2}} \cong \spechtd{1}$.

So in this case the rank of $W_{2,n}^{1}(m)$ is 
\[
\binom{m}{1} - \binom{m}{0} = m-1.
\]

\subsection{Case: $\ch{F} = p$, $p>2$.}\hfil \\

In these remaining cases we see from Lemma \ref{lem:maps1} part $(3)$ that $\im{\rho_{n,2}}$ contains the Specht module $\specht{2}$; that is, $L^{2} \cong \specht{2}$.  

\subsubsection{Case: $p \nmid m-2n$ and $p \nmid n(m-n)$.}\hfil \\

From Lemma \ref{lem:maps1} we see that $L^{0}$ and $L^{1}$ are also isomorphic to the full Specht modules.  Thus in this case the rank of $W_{2,n}^{1}(m)$ is 
\[
\binom{m}{2} = \frac{m(m-1)}{2}.
\]

\subsubsection{Case: $p \nmid m-2n$ and $p \mid n(m-n)$.}\hfil \\

By Lemma \ref{lem:maps1} we see that $L^{0} = \{0\}$ and $L^{1}$ is isomorphic to the full Specht module.  We have that the rank of $W_{2,n}^{1}(m)$ equals
\[
\binom{m}{2} - 1 = \frac{(m+1)(m-2)}{2}.
\]

\subsubsection{Case: $p \mid m-2n$.}\hfil \\

This rank in this case will depend on whether or not $p$ divides $m$.

\paragraph{Case:  $p \mid m$.} \hfil \\

We must have that $p \mid n$ as well, so by Lemma \ref{lem:maps1} part $(1)$ $L^{0} = \{0\}$.  Part $(2)$ of Lemma \ref{lem:maps2} shows that $\psi_{2,1}$ kills $\im{\rho_{n,2}}$, so we have $L^{1} = \{0\}$.  Thus the rank of $W_{2,n}^{1}(m)$ is
\[
\binom{m}{2} - \binom{m}{1} = \frac{m(m-3)}{2}.
\]

\paragraph{Case:  $p \nmid m$.} \hfil \\

This implies that $p \nmid n$, so by Lemma \ref{lem:maps1} part $(1)$ we see that $L^{0} \cong D^{(m)}$.  By Lemma \ref{lem:maps2} part $(2)$ we see that $\psi_{2,1}$ maps $\im{\rho_{n,2}}$ into $\langle \allone \rangle$.  Thus we must have that $L^{1} = \{0\}$ or $L^{1} \cong D^{(m)}$.  But Theorem \ref{thm:james} shows that $D^{(m)}$ is not a composition factor of $\specht{1}$ in this case.  So we have $L^{1} = \{0\}$, and the rank of $W_{2,n}^{1}(m)$ is
\[
\binom{m}{2} - \binom{m}{1} + 1 = \frac{(m-1)(m-2)}{2}.
\]

\section{General subset intersection matrices}\label{sec:conclude}
We conclude with some observations about the general incidence matrix $W_{k,n}^{i}(m)$ mentioned in the introduction.  This matrix represents the map $\tau_{n,k}^{i} \colon \Mperm{n} \to \Mperm{k}$ which sends an $n$-subset of $\XX$ to the sum of the $k$-subsets that intersect it in a set of size $i$.  Our often used Lemma \ref{lem:maps1} readily generalizes to this situation.

\begin{lemma}\label{lem:maps3}
Let $0 \leq i \leq k \leq n \leq \frac{1}{2}m$.  Let $t$ be an $(m-n,n)$-tableau.  Then, for $j \leq k$, $\tau_{n,k}^{i}$ maps any $j$-polytabloid $e_{t}^{j}$ in $\Mperm{n}$ into $\spechtj{k}{j}$.  Furthermore,
\[
\psi_{k,j}\left(\tau_{n,k}^{i}(e_{t}^{j})\right) = \left(\sum_{\ell \geq 0} (-1)^{j-\ell} \binom{j}{\ell} \binom{n-j}{i-\ell} \binom{m-n-j}{k-i-j+\ell}\right) e_{s},
\]
where $s$ is the $(m-j,j)$-tableau obtained from $t$ by moving the last $n-j$ entries of the second row of $t$ to the end of the first row of $t$.
\end{lemma}
\begin{proof}
Assume the hypotheses of the lemma and let $t$ be an $(m-n,n)$-tableau.  Then
\[
e_{t}^{j} \xmapsto{\tau_{n,k}^{i}} \kappa_{t}^{j}\sum_{\vert \{s\} \cap \{t\}\vert = i}\{s\}.
\]
The terms in the sum above are $k$-subsets $\{s\}$ that meet $\{t\}$ in $i$ elements of the second row of $t$; thus the other $k-i$ elements are chosen from the first row of $t$.  The terms in the sum that are not killed by $\kappa_{t}^{j}$ are precisely those $k$-subsets that contain exactly one entry from each of the first $j$ columns of $t$.  We may classify such a $k$-subset by how many of those entries from the first $j$ columns of $t$ came from the second row.  Call this number $\ell$.  Then we see the number of terms in the sum not killed by $\kappa_{t}^{j}$ is
\[
\sum_{\ell \geq 0} \binom{j}{\ell} \binom{n-j}{i-\ell} \binom{m-n-j}{k-i-j+\ell}.
\]
Each $k$-subset that is such a term in the sum may be represented by an $(m-k,k)$-tableau $r$ that has the same first $j$ columns as $t$, where the entries of each of the first $j$ columns may be transposed (so that elements of the $k$-subset are in the second row).  Applying $\kappa_{t}^{j}$ to each of these $k$-subsets $\{r\}$ we in fact then have a sum of $j$-polytabloids $e_{r}^{j}$ in $\spechtj{k}{j}$. 

Finally, the inclusion map $\psi_{k,j}$ now sends any such $j$-polytabloid $e_{r}^{j}$ to $\pm e_{s}$, where $s$ is the $(m-j,j)$-tableau described in the statement of the lemma.  The sign depends only on the number of transpositions $j-\ell$ that change the first $j$ columns of $r$ into those of $t$.  The lemma follows.
\end{proof}

Using this lemma we could begin a similar analysis of any of the subset incidence matrices as we did for $W_{2,n}^{1}(m)$.  We can also get some easy results such as: 
\begin{corollary}
Let $0 \leq i \leq k \leq n \leq \frac{1}{2}m$.  If $\ch{F}=0$, or $\ch{F}=p$ with $p \nmid \binom{k}{i}$, then the matrix $W_{k,n}^{i}(m)$ has rank at least
\[
\binom{m}{k} - \binom{m}{k-1}.
\]
\end{corollary}
\begin{proof}
Let $t$ be an $(m-n,n)$-tableau.  Applying Lemma \ref{lem:maps3} to the case when $j=k$, we see that the $k$-polytabloid $e_{t}^{k}$ is mapped by the $FG_{m}$-module homomorphism $\tau_{n,k}^{i}$ to $\pm\binom{k}{i}e_{s}$, where the tableau $s$ is described in the lemma.  Thus, if $\binom{k}{i} \neq 0$, $\im{\tau_{n,k}^{i}}$ contains $\specht{k}$ and the corollary follows. 
\end{proof}

We leave the reader with a final remark, which we hope will encourage more study of these problems through the lens of representation theory.  Notice that for the inclusion relation, where $i=k$, the coefficients of $e_{s}$ in Lemma \ref{lem:maps3} are
\[
\binom{n-j}{k-j},\]
for $0 \leq j \leq k$.  These same numbers appear in Wilson's diagonal form for the inclusion matrices, with multiplicities equal to Specht module dimensions.  Furthermore, the coefficients of $e_{s}$ appearing in the case of $1$-intersection of $2$-subsets vs. $n$-subsets, which can be read from Lemma \ref{lem:maps1} (or extracted from Lemma \ref{lem:maps3})  are very similar to the numbers appearing in the diagonal form for $W_{2,n}^{1}(m)$ given in Theorem 17 of \cite{wong}.  Looking at various examples, one can check that these coefficients (with multiplicities equal to Specht module dimensions) do in fact give an alternative diagonal form for this matrix in some cases (but not all).

%

\bibliographystyle{amsplain}
\bibliography{subset_bib}
\end{document}